\PassOptionsToPackage{table}{xcolor}
\documentclass[11pt]{article} 
\usepackage{fullpage}

\usepackage{times}
\usepackage{tikz}
\usetikzlibrary{arrows.meta, positioning, calc}
\usepackage{mathtools}
\usepackage{amsthm, amsmath, amssymb}
\usepackage{dsfont}
\usepackage{verbatim}
\usepackage{enumitem}
\usepackage{xspace}
\usepackage{todonotes}
\usepackage{amsfonts,mathrsfs,color}
\usepackage{xcolor}
\usepackage{graphicx}
\usepackage{booktabs}
\usepackage{nicefrac}
\usepackage{mdframed}
\usepackage{contour}
\contourlength{0.1pt}
\contournumber{10}
\usepackage[
  backend=biber,
  style=alphabetic,
  maxbibnames=10,
  maxalphanames=3,
  minalphanames=3,
  natbib=true,
  url=false
]{biblatex}
\usepackage{multirow}
\usepackage[ruled,linesnumbered,noend]{algorithm2e} 
\SetKwComment{Comment}{/* }{ */}
\SetKwProg{Fn}{function}
\SetAlFnt{\relax}
\SetAlCapFnt{\relax}
\SetAlCapNameFnt{\relax}
\SetAlCapHSkip{0pt}
\usepackage{tikz}
\usepackage{xcolor}

\newcommand{\declarecolor}[2]{\definecolor{#1}{RGB}{#2}\expandafter\newcommand\csname #1\endcsname[1]{\textcolor{#1}{##1}}}
\declarecolor{White}{255, 255, 255}
\declarecolor{Black}{0, 0, 0}
\declarecolor{Maroon}{128, 0, 0}
\declarecolor{Coral}{255, 127, 80}
\declarecolor{Red}{182, 21, 21}
\declarecolor{LimeGreen}{50, 205, 50}
\declarecolor{DarkGreen}{0, 80, 0}
\declarecolor{Purple}{146, 42, 158}
\declarecolor{Navy}{0, 0, 128}
\declarecolor{LightBlue}{84, 101, 202}
\usepackage{hyperref}
\hypersetup{
    colorlinks,
    citecolor=DarkGreen,
    linkcolor=LightBlue}
\usepackage[noabbrev, capitalise, nameinlink]{cleveref}

\newtheorem{theorem}{Theorem}
\newtheorem*{theorem*}{Theorem}
\newtheorem{lemma}{Lemma}

\newtheorem{example}{Example}

\newtheorem{proposition}{Proposition}

\usepackage{pifont}

\def\+#1{\mathcal{#1}}
\def\-#1{\mathbb{#1}}

\newcommand{\notshow}[1]{{}}
\newcommand{\AutoAdjust}[3]{{ \mathchoice{ \left #1 #2  \right #3}{#1 #2 #3}{#1 #2 #3}{#1 #2 #3} }}
\newcommand{\Xcomment}[1]{{}}

\newcommand{\InParentheses}[1]{\AutoAdjust{(}{#1}{)}}
\newcommand{\InBrackets}[1]{\AutoAdjust{[}{#1}{]}}

\newcommand{\InNorms}[1]{\AutoAdjust{\|}{#1}{\|}}
\renewcommand{\part}[2]{\frac{\partial #1}{\partial #2}}

\newcommand{\R}{\mathds{R}}

\newcommand{\prox}{\mathbf{Prox}}

\newcommand{\Id}{\mathrm{Id}}
\newcommand{\zer}{\operatorname{zer}}

\newcommand{\rtan}{r^{\mathrm{tan}}}
\newcommand{\rnat}{r^{\mathrm{nat}}}
\newcommand{\ip}[2]{\left\langle #1, #2 \right\rangle}
\newcommand{\norm}[1]{\left\lVert #1 \right\rVert}
\newcommand{\set}[1]{\left\{ #1 \right\}}

\title{Last-Iterate Convergence of Anchored Gradient Descent}
\author{Yang Cai\thanks{Authors are alphabetically ordered.} \\ Yale University\\ yang.cai@yale.edu
\and Weiqiang Zheng\footnotemark[1] \\ Yale University \\ weiqiang.zheng@yale.edu}

\begin{document}

\maketitle

\begin{abstract}%
We study the monotone inclusion problem $0\in F(z)+A(z)$, where $F$ is monotone and Lipschitz, and $A$ is maximally monotone, a framework that encompasses monotone variational inequalities and convex-concave saddle-point problems with constraints or regularization. It is well known that vanilla gradient descent diverges for this problem, whereas optimism-based methods such as Extragradient and accelerated methods that combine both optimism and anchoring, such as Extra Anchored Gradient, achieve last-iterate convergence. However, the anchoring-only method, anchored gradient descent, has been studied only in the unconstrained setting~\citep{ryu2019ode, surina2026anchored}. In this note, we extend the anchored gradient descent method to the monotone inclusion problem and prove a last-iterate convergence rate of $O(1/\sqrt{T})$ in terms of the tangent residual. We build on the recent proof in the unconstrained setting~\citep{surina2026anchored} and use techniques from~\citep{cai2024accelerated} to extend it to the general inclusion setting. 
\end{abstract}

\setcounter{tocdepth}{3}
\tableofcontents
\addtocounter{page}{-1}
\thispagestyle{empty}

\newpage
\section{Introduction}
In this paper, we focus on the monotone inclusion problem
\begin{equation}\label{eq:intro-mi}
0\in F(z)+A(z),
\end{equation}
where $F:\R^d\to\R^d$ is a single-valued Lipschitz and monotone operator and $A:\R^d\rightrightarrows \R^d$ is maximally monotone. The monotone inclusion problem is general and important in optimization, machine learning, and game theory. The monotone inclusion problem covers the monotone variational inequality problem and the convex-concave min-max optimization problem as special cases.
\begin{example}[Monotone variational inequality]\label{ex:vi}
Let $\+Z\subseteq \R^d$ be closed and convex, and $A=N_{\+Z}$ be the normal cone operator for $\+Z$. Then $z^\star \in \+Z$ solves
\[
0\in F(z^\star)+N_{\+Z}(z^\star)
\]
if and only if 
\[
\ip{F(z^\star)}{z-z^\star}\ge 0\qquad \forall z\in \+Z.
\]
Thus monotone inclusion recovers the standard (Stampacchia) variational inequality problem.
\end{example}

\begin{example}[Convex-concave min-max problems]\label{ex:minmax}
Let $\Phi:\R^n\times \R^m\to \R$ be continuously differentiable, convex in $x$, and concave in $y$. Let $g:\R^n\to (-\infty,+\infty]$ and $h:\R^m\to (-\infty,+\infty]$ be proper closed convex functions, and consider
\[
\min_{x\in \R^n}\max_{y\in \R^m}\; \Phi(x,y)+g(x)-h(y).
\]
Define $z=(x,y)$,
\[
F(z):=\bigl(\nabla_x\Phi(x,y),-\nabla_y\Phi(x,y)\bigr),
\qquad
A(z):=\partial g(x)\times \partial h(y).
\]
Then $F$ is monotone by convexity-concavity, and $(x^\star,y^\star)$ is a saddle point if and only if
\[
0\in F(x^\star,y^\star)+A(x^\star,y^\star).
\]
The special choice $g=I_{\+X}$ and $h=I_{\+Y}$ gives the constrained saddle problem over $\+X\times \+Y$.
\end{example}
We measure the optimality using the tangent residual~\citep{cai2022finite}:
\begin{equation*}
    \rtan_{F,A}(z):= \mathrm{Distance}(0, F(z) + A(z)) = \inf_{c\in A(z)}\norm{F(z)+c}.
\end{equation*}
The tangent residual upper bounds the natural residual and implies bounds on the gap function. In the unconstrained setting $A = 0$, the tangent residual is just the operator norm $\InNorms{F(z)}$.

\paragraph{Gradient Descent} Simultaneous gradient descent (GD) algorithm is a fundamental first-order method (also known as the forward-backward splitting algorithm) for monotone inclusion problems
\begin{equation}\label{GD}\tag{GD}
    \begin{aligned}
        z_{t+1}=J_{\alpha_t A}[z_t-\alpha_tF(z_t)].
    \end{aligned}
\end{equation}
Unfortunately, the vanilla GD algorithm is well-known to oscillate and fail to converge even on very simple bilinear min-max problems~\citep{daskalakis2018limit}. 

\paragraph{Remedies for Last-Iterate Convergence: Optimism and Anchoring} There are two remedies of vanilla GD that ensure last-iterate convergence: extragradient-type optimism and anchoring. 

The extragradient (EG) method~\citep{korpelevich1976extragradient} uses a look-ahead point to correct the gradient update:
\begin{equation}\label{EG}\tag{EG}
    \begin{aligned}
    z_{t+\frac{1}{2}} &= J_{\alpha_t A}[z_t - \alpha_t F(z_t)],
    \\
    z_{t+1} &= J_{\alpha_t A}[z_t - \alpha_t F(z_{t+\frac{1}{2}})].
\end{aligned}
\end{equation}
The EG method has $O(\frac{1}{\sqrt{T}})$ last-iterate convergence rate in terms of the tangent residual for constrained monotone variational inequalities (\Cref{ex:vi}, $A = N_{\+Z}$), proved by \citet{cai2022finite}. Their analysis can be extended to the case where $A = \partial g$ is the subdifferential operator of a convex function. We note that the same convergence rate also holds for the slightly more general case where $A$ is 3-cyclically monotone~\citep{tran2026revisiting}, which covers the above two cases. The EG algorithm has been studied in the unconstrained setting~\citep{golowich2020last, gorbunov2022extragradient}. The optimistic gradient descent ascent (OGDA) algorithm~\citep{popov1980modification} also uses the optimism mechanism and is similar to EG. OGDA has $O(1/\sqrt{T})$ last-iterate convergence rate, first proved in the unconstrained setting~\citep{golowich2020tight} and then the constrained setting~\citep{cai2022finite, gorbunov2022last}.

Another mechanism is anchoring, first proposed in Halpern iteration~\citep{halpern1967fixed}, which adds regularization toward a fixed anchor point $z_0$. Applying anchoring to \ref{GD} leads to the following anchored GD \eqref{AGD} method, first proposed by ~\citet{ryu2019ode} in the unconstrained setting ($A = 0$):
\begin{equation}\label{AGD}\tag{AGD}
    \begin{aligned}
        z_{t+1}=z_t-\alpha_tF(z_t)+\beta_t(z_0-z_t).
    \end{aligned}
\end{equation}
\citet{ryu2019ode} propose a series of $\{\alpha_t, \beta_t\}$ that achieve last-iterate convergence rate of $\InNorms{F(z_T)} = O(T^{-\frac{1}{2}+p})$ for any $p \in (0, \frac{1}{2})$. Importantly, this leaves open whether $O(\frac{1}{\sqrt{T}})$ last-iterate convergence rate is possible. A very recent paper by \citet{surina2026anchored} prove an improved $\InNorms{F(z_T)} = O(\frac{1}{\sqrt{T}})$ last-iterate convergence rate for \ref{AGD} under the schedule
\begin{align}\label{eq:step size}
    \alpha_t=\frac{1}{L\sqrt{t+\gamma}},
    \qquad
    \beta_t=\frac{\gamma}{t+\gamma},
    \qquad \gamma\ge 2.
\end{align}

However, all these results hold only in the unconstrained setting where $A = 0$. It remains unknown whether the \ref{AGD} can be generalized to general monotone inclusion in the composite setting. \citet{ryu2019ode} explicitly mentions this question: ``\textit{Generalizing the results of this work to accommodate projections and proximal operators, analogous to projected and proximal gradient methods, is an interesting direction of future work.}" 

\subsection{Our Results}
In this note, we answer the question by \citep{ryu2019ode} in the affirmative. 

\paragraph{Proximal Anchored Gradient Descent.} We propose the proximal anchored gradient descent method \eqref{P-AGD} for the composite monotone inclusion problem \eqref{eq:intro-mi}:
\begin{equation}\label{P-AGD}\tag{P-AGD}
\begin{aligned}
    z_{t+1} &= J_{\alpha_t A} \InBrackets{z_t - \alpha_t F(z_t) + \beta_t(z_0-z_t)}.
\end{aligned}
\end{equation}
where $J_{\alpha_t A}$ is the resolvent operator (see \Cref{sec:resolvents} for definition) that covers both projections and proximal operators. We use the same schedule in \eqref{eq:step size} as \citet{surina2026anchored}:

We show that \ref{P-AGD} achieves the $O(1/\sqrt{T})$ last-iterate convergence rate for the tangent residual for the composite monotone inclusion problem \eqref{eq:intro-mi}.
\begin{theorem}\label{thm:main-intro}
    Let $F$ be $L$-Lipschitz and monotone, $A$ be maximally monotone, and $z^\star$ be any solution to the monotone inclusion problem. Let $z_0\in\R^d$ and $\gamma\ge 2$. Then for the sequence generated by \ref{P-AGD} with step sizes \eqref{eq:step size}, we have for all $T \ge 1$,
    \begin{align*}
        \rtan_{F,A}(z_T) = \inf_{c\in A(z_{T})}\norm{F(z_T)+c} \le \frac{25(\sqrt{12}+1)L\gamma\InNorms{z_0 - z^\star}}{\sqrt{T-1+\gamma}} = O\InParentheses{\frac{L \InNorms{z_0-z^\star}}{\sqrt{T}}}.
    \end{align*}
\end{theorem}
A few remarks are in order.

\paragraph{Optimal Convergence Rate.} We note the $O(\frac{1}{\sqrt{T}})$ is not the optimal convergence rate for monotone inclusion problems. By combining optimism and anchoring, we can achieve an accelerated, optimal convergence rate of $O(\frac{1}{T})$. This is first proved in the unconstrained setting for the Extra Anchored Gradient (EAG) algorithm by \citet{yoon2021accelerated} and the Fast ExtraGradient (FEG) algorithm by~\citet{lee2021fast}. \citet{yoon2021accelerated} also prove a matching lower bound of $\Omega(\frac{1}{T})$.
The optimal $O(\frac{1}{T})$ last-iterate convergence rate for the general monotone inclusion problem is first achieved by~\citet{cai2024accelerated} through the composite generalization of the EAG and the FEG algorithms, and concurrently by~\citet{kovalev2022first} by the composite generalization of EAG. Together with the last-iterate convergence rates for \ref{EG} that uses optimism only and our results for \ref{P-AGD} that uses anchoring only, these results illustrate the full landscape of acceleration mechanisms for gradient descent. We summarize this landscape in \Cref{fig:landscape}. Accelerated algorithms combining both optimism and anchoring have also been explored in~\citep{bot2022fast, cai2023accelerated, cai2023doubly, bot2023fast, tran2024halpern}.

\begin{figure}[ht]
\centering
\begin{tikzpicture}[
    alg/.style={
        rectangle, rounded corners=6pt,
        draw=black!80, line width=0.8pt,
        minimum width=5.0cm, minimum height=1.1cm,
        text width=4.8cm, align=center,
        font=\small
    },
    rate/.style={
        font=\footnotesize\itshape, text=black!70
    },
    mechanism/.style={
        font=\footnotesize\bfseries, text=Purple,
        fill=white, inner sep=2pt
    },
    arrowstyle/.style={
        -{Stealth[length=6pt, width=4pt]},
        line width=0.7pt, black!70
    },
    level/.style={
        draw=black!20, dashed, line width=0.5pt
    }
]

\node[alg, fill=red!8] (GD) at (0, 0)
    {\textbf{Gradient Descent}\\[2pt] \ref{GD}};

\node[alg, fill=blue!8] (EG) at (-3.5, -3)
    {\textbf{Extragradient}\\[2pt] \ref{EG}~\citep{korpelevich1976extragradient, cai2022finite}};

\node[alg, fill=orange!8] (AGD) at (3.5, -3)
    {\textbf{Anchored GD}\\[2pt] \ref{AGD}~\citep{ryu2019ode,surina2026anchored}\\ [2pt] \ref{P-AGD} [\textbf{This work}]};

\node[alg, fill=green!8] (EAG) at (0, -6)
    {\textbf{Extra Anchored Gradient}\\[2pt] EAG~\citep{yoon2021accelerated}\\ [2pt] composite-EAG~\citep{cai2024accelerated, kovalev2022first}};

\draw[arrowstyle] (GD.south) -- (EG.north);
\draw[arrowstyle] (GD.south) -- (AGD.north);
\draw[arrowstyle] (EG.south) -- (EAG.north);
\draw[arrowstyle] (AGD.south) -- (EAG.north);

\node[mechanism, text=blue!70!black] at ($(GD.south)!0.5!(EG.north)$) [left=3pt]
    {+ Optimism};
\node[mechanism, text=orange!70!black] at ($(GD.south)!0.5!(AGD.north)$) [right=3pt]
    {+ Anchoring};
\node[mechanism, text=orange!70!black] at ($(EG.south)!0.5!(EAG.north)$) [left=3pt]
    {+ Anchoring};
\node[mechanism, text=blue!70!black] at ($(AGD.south)!0.5!(EAG.north)$) [right=3pt]
    {+ Optimism};

\node[rate, text=red!70!black] at (7.2, 0) {\textbf{Diverges}};
\node[rate, text=blue!50!black] at (7.2, -3) {$O(1/\sqrt{T})$ rate};
\node[rate, text=green!40!black] at (7.2, -6) {$O(1/T)$ rate \textbf{(optimal)}};

\draw[level] (-6, -1.5) -- (8, -1.5);
\draw[level] (-6, -4.5) -- (8, -4.5);

\end{tikzpicture}

\caption{Landscape of last-iterate convergence rate of acceleration mechanisms for gradient descent \eqref{GD} in monotone inclusion. Vanilla \ref{GD} diverges. Adding \emph{optimism} alone yields the Extragradient method (\ref{EG},~\citep{korpelevich1976extragradient}), while adding \emph{anchoring} alone yields Anchored GD (\ref{AGD}, \citep{ryu2019ode}) and proximal anchored GD (\ref{P-AGD}); both achieve a $O(1/\sqrt{T})$ last-iterate rate~\citep{cai2022finite, surina2026anchored} (the result for \ref{EG} holds when $A = \partial g$ is the subdifferential operator of a convex function). Combining both mechanisms yields optimal $O(1/T)$ methods such as the Extra Anchored Gradient (EAG,~\citep{yoon2021accelerated}) and the Fast Extra Gradient (FEG,~\citep{lee2021fast}) and their composite extensions~\citep{cai2024accelerated, kovalev2022first}.}
\label{fig:landscape}
\end{figure}

\paragraph{Proof Idea of \Cref{thm:main-intro}} Our starting point is the proof in the unconstrained setting by \citet{surina2026anchored}. The main observation is that we can use the same idea introduced in \citep{cai2024accelerated} to handle the composite operator $A$ by introducing auxiliary vectors $c_{t+1} \in A(z_{t+1})$ defined as $c_{t+1}:=(z_t -\alpha_t F(z_t) + \beta_t(z_0 -z_t) - z_{t+1})/\alpha_t$. With $c_{t+1}$, the update of \ref{P-AGD} becomes unconstrained: $z_{t+1} = z_t -\alpha_t (F(z_t)+c_{t+1}) + \beta_t(z_0 -z_t)$. Then the proof by \citet{surina2026anchored} in the unconstrained setting ($A = 0$) extends nicely to the general monotone inclusion problem, with properties of the resolvent operator. We upper bound $\InNorms{F(z_t) + c_t}$, which then upper bounds the tangent residual $\rtan_{F,A}(z_t) = \inf_{c\in A(z_t)}\norm{F(z_t)+c}$.

\section{Preliminaries}
We work in the Euclidean space $\R^d$ with norm $\norm{\cdot}$ induced by the inner product $\ip{\cdot}{\cdot}$. Throughout, we study the monotone inclusion problem~\eqref{eq:intro-mi} and assume:
\begin{enumerate}[label=(A\arabic*),leftmargin=2.5em]
    \item $F:\R^d\to\R^d$ is monotone, i.e., $\ip{F(z)-F(w)}{z-w}\ge 0$ for all $z,w\in \R^d$.
    \item $F$ is $L$-Lipschitz, i.e., $\norm{F(z)-F(w)}\le L\norm{z-w}$ for all $z,w\in \R^d$.
    \item $A:\R^d\rightrightarrows \R^d$ is maximally monotone, i.e., $A$ is monotone, and its graph is not a strict subset of another monotone operator.
    \item The solution set $S:=\zer(F+A)=\set{z:0\in F(z)+A(z)}$ is nonempty.
\end{enumerate}

\subsection{Resolvents, Normal Cones, and Prox Operators}\label{sec:resolvents}
For $\alpha>0$, the resolvent of $A$ is
\[
J_{\alpha A}:=(\Id+\alpha A)^{-1}.
\]
The following properties are standard; see, for example, \citep{ryu2022large}.

\begin{proposition}[Facts about the resolvent  operator]\label{prop:resolvent} A set-valued operator $A:\R^d\rightrightarrows \R^d$ satisfyies
\begin{enumerate}[label=(\roman*),leftmargin=2em]
    \item If $A$ is maximally monotone, then $J_{\alpha A}$ is single-valued and non-expansive;
    \item $z=J_{\alpha A}(w)$ if and only if $\dfrac{w-z}{\alpha}\in A(z)$;
    \item if $A=N_{\+Z}$ for a closed convex set $\+Z$, then $A$ is maximally monotone and $J_{\alpha A}=\Pi_{\+Z}$ is projection onto $\+Z$;
    \item if $A=\partial g$ for a proper closed convex function $g$, then $A$ is maximally monotone and $J_{\alpha A}=\prox_{\alpha g}$ is the proximal operator~\citep{parikh2014proximal}.
\end{enumerate}
\end{proposition}

\subsection{Residuals}
The right convergence measure for the composite problem is the tangent residual
\[
\rtan_{F,A}(z):=\min_{c\in A(z)}\norm{F(z)+c}.
\]
This reduces to $\norm{F(z)}$ when $A\equiv 0$, and it is exactly the condition for approximate stationarity:
\[
\rtan_{F,A}(z)\le \varepsilon
\quad\Longleftrightarrow\quad
0\in F(z)+A(z)+B(0,\varepsilon).
\]
It is also useful to recall the natural residual
\[
\rnat_{F,A}(z):=\norm{z-J_A(z-F(z))}.
\]
The natural residual is always dominated by the tangent residual.

\begin{proposition}\label{prop:nat-tan}
For every $z\in \R^d$,
\[
\rnat_{F,A}(z)\le \rtan_{F,A}(z).
\]
\end{proposition}

\begin{proof}
Fix $z\in \R^d$ and $c\in A(z)$. Since $z=J_A(z+c)$ by Proposition~\ref{prop:resolvent}(ii), the nonexpansiveness of $J_A$ gives
\[
\rnat_{F,A}(z)
=\norm{J_A(z+c)-J_A(z-F(z))}
\le \norm{F(z)+c}.
\]
Taking the minimum over $c\in A(z)$ proves the claim.
\end{proof}

\section{Main Result}
We propose the following generalization of the \ref{AGD} algorithm for monotone inclusion problems \eqref{eq:intro-mi}.
\begin{equation}\label{eq:alg}
z_{t+1}=J_{\alpha_tA}\Bigl((1-\beta_t)z_t+\beta_tz_0-\alpha_tF(z_t)\Bigr),
\qquad t\ge 0,
\end{equation}
with the schedule
\begin{equation}\label{eq:schedule}
\alpha_t=\frac{1}{L\sqrt{t+\gamma}},
\qquad
\beta_t=\frac{\gamma}{t+\gamma},
\qquad \gamma\ge 2.
\end{equation}
When $A=N_{\+Z}$, this is the projected method
\[
z_{t+1}=\Pi_{\+Z}\Bigl((1-\beta_t)z_t+\beta_tz_0-\alpha_tF(z_t)\Bigr),
\]
and when $A=\partial g\times \partial h$ in \Cref{ex:minmax}, it becomes the proximal update
\[
\begin{aligned}
x_{t+1}&=\prox_{\alpha_t g}\Bigl((1-\beta_t)x_t+\beta_t x_0-\alpha_t\nabla_x\Phi(x_t,y_t)\Bigr),\\
y_{t+1}&=\prox_{\alpha_t h}\Bigl((1-\beta_t)y_t+\beta_t y_0+\alpha_t\nabla_y\Phi(x_t,y_t)\Bigr).
\end{aligned}
\]

The key quantity for the proof is the vector
\begin{equation}\label{eq:def-c}
c_{t+1}:=\frac{(1-\beta_t)z_t+\beta_tz_0-\alpha_tF(z_t)-z_{t+1}}{\alpha_t},\quad \forall t \ge 0
\end{equation}
By Proposition~\ref{prop:resolvent}(ii), we have $c_{t+1}\in A(z_{t+1})$. This is the $c$-vector trick of~\citep{cai2024accelerated}.

\begin{theorem}[Last-iterate $O(1/\sqrt{T})$ rate in  tangent residual]\label{thm:main}
Assume \textup{(A1)}--\textup{(A4)}, and let $\{z_t\}_{t\ge 0}$ be generated by \eqref{eq:alg}--\eqref{eq:schedule}. Fix any solution $z^\star\in S$, and define
\[
H_0:=\norm{z_0-z^\star},
\qquad
D:=(\sqrt{12}+1)H_0,
\qquad
E:=\max\!\bigl\{\gamma\norm{z_1-z_0},\,12\gamma D\bigr\}.
\]
Then for every $T\ge 1$,
\begin{equation}\label{eq:main-bound}
\rtan_{F,A}(z_T)
\le
\norm{F(z_T)+c_T}
\le
\frac{L(2E+\gamma D)}{\sqrt{T-1+\gamma}}.
\end{equation}
In particular, since $\norm{z_1-z_0}\le D$, one may take $E=12\gamma D$, and therefore
\begin{equation}\label{eq:main-explicit}
\rtan_{F,A}(z_T)
\le
\frac{25\,\gamma L D}{\sqrt{T-1+\gamma}}
=
O\!\left(\frac{1}{\sqrt{T}}\right).
\end{equation}
Consequently, the same $O(1/\sqrt{T})$ bound holds for the natural residual $\rnat_{F,A}(z_T)$.
\end{theorem}

The proof follows the same three-step template as the unconstrained ($A = 0$) proof of \citet{surina2026anchored}: first bound the iterates, then bound consecutive differences, and finally convert that estimate into a last-iterate residual bound. The new ingredient is handling the composite term $A$ using the vectors $\{c_t\in A(z_t)\}$ and properties of the resolvent operator.

\section{Proof of the Main Theorem}
We fix a solution $z^\star\in S$ throughout. Choose $a^\star\in A(z^\star)$ such that $F(z^\star)+a^\star=0$.

\subsection{Step 1: Bounded Iterates}
\begin{lemma}\label{lem:bounded-iterates}
For all $t\ge 0$,
\[
\norm{z_t-z^\star}^2\le 12H_0^2.
\]
Consequently,
\[
\norm{z_t-z_0}\le D
\qquad\text{for all }t\ge 0.
\]
\end{lemma}

\begin{proof}
Since $F(z^\star)+a^\star=0$, Proposition~\ref{prop:resolvent}(ii) implies
\[
z^\star=J_{\alpha_tA}\bigl(z^\star-\alpha_tF(z^\star)\bigr)
\qquad \forall t\ge 0.
\]
Let
\[
u_t:=z_t-z^\star,
\qquad
v:=z_0-z^\star,
\qquad
g_t:=F(z_t)-F(z^\star),
\qquad
w_t:=(1-\beta_t)u_t+\beta_tv.
\]
By the nonexpansiveness of $J_{\alpha_tA}$,
\[
\norm{u_{t+1}}\le \norm{w_t-\alpha_tg_t}.
\]
Squaring both sides gives
\[
\norm{u_{t+1}}^2
\le
\norm{w_t}^2-2\alpha_t\ip{w_t}{g_t}+\alpha_t^2\norm{g_t}^2.
\]
Now
\[
-2\alpha_t\ip{w_t}{g_t}
=
-2\alpha_t(1-\beta_t)\ip{u_t}{g_t}-2\alpha_t\beta_t\ip{v}{g_t}.
\]
By monotonicity of $F$,
\[
\ip{u_t}{g_t}=\ip{z_t-z^\star}{F(z_t)-F(z^\star)}\ge 0,
\]
so the first term is nonpositive. For the second term, use
\[
0\le \norm{2\beta_tv+\alpha_tg_t}^2
=4\beta_t^2\norm{v}^2+4\alpha_t\beta_t\ip{v}{g_t}+\alpha_t^2\norm{g_t}^2,
\]
which implies
\[
-2\alpha_t\beta_t\ip{v}{g_t}
\le 2\beta_t^2\norm{v}^2+\frac12\alpha_t^2\norm{g_t}^2.
\]
Also, by the convexity of the squared norm,
\[
\norm{w_t}^2\le (1-\beta_t)\norm{u_t}^2+\beta_t\norm{v}^2.
\]
Combining these estimates and using $\norm{g_t}\le L\norm{u_t}$ gives
\begin{equation}\label{eq:bound-u-recurrence}
\norm{u_{t+1}}^2
\le
\bigl(1-\beta_t+\tfrac32\alpha_t^2L^2\bigr)\norm{u_t}^2
+
\bigl(\beta_t+2\beta_t^2\bigr)\norm{v}^2.
\end{equation}
We now prove by induction that $\norm{u_t}^2\le 12\norm{v}^2=12H_0^2$ for all $t\ge 0$. The claim is obvious at $t=0$. Assuming it holds at time $t$, \eqref{eq:bound-u-recurrence} yields
\[
\norm{u_{t+1}}^2
\le
\Bigl(12-11\beta_t+18\alpha_t^2L^2+2\beta_t^2\Bigr)H_0^2.
\]
Because
\[
\alpha_t^2L^2=\frac{1}{t+\gamma},
\qquad
\beta_t=\frac{\gamma}{t+\gamma},
\]
we obtain
\[
-11\beta_t+18\alpha_t^2L^2+2\beta_t^2
=
\frac{(18-11\gamma)t+\gamma(18-9\gamma)}{(t+\gamma)^2}
\le 0
\]
for every $\gamma\ge 2$. Hence $\norm{u_{t+1}}^2\le 12H_0^2$, completing the induction.

The second claim follows from the triangle inequality:
\[
\norm{z_t-z_0}\le \norm{z_t-z^\star}+\norm{z^\star-z_0}\le (\sqrt{12}+1)H_0=D.\qedhere
\]
\end{proof}

\subsection{Step 2: One-Step Recursion and Stability of Iterates}
Define the consecutive differences
\[
d_t:=z_{t+1}-z_t,
\qquad t\ge 0.
\]
From \eqref{eq:def-c}, the algorithm can be rewritten as
\begin{equation}\label{eq:update-c}
d_t=\beta_t(z_0-z_t)-\alpha_t\bigl(F(z_t)+c_{t+1}\bigr).
\end{equation}
This is the exact composite analogue of the unconstrained Anchored GDA recursion.

\begin{lemma}\label{lem:d-recurrence}
For every $t\ge 0$,
\[
\norm{d_{t+1}}\le q_t\norm{d_t}+|\varepsilon_t|D,
\]
where
\[
\lambda_t:=\frac{\alpha_{t+1}}{\alpha_t}-\beta_{t+1},
\qquad
\varepsilon_t:=\beta_{t+1}-\frac{\alpha_{t+1}}{\alpha_t}\beta_t,
\qquad
q_t:=\sqrt{\lambda_t^2+\alpha_{t+1}^2L^2}.
\]
\end{lemma}

\begin{proof}
Set
\[
a_t:=(1-\beta_t)z_t+\beta_tz_0-\alpha_tF(z_t),
\]
so that $z_{t+1}=J_{\alpha_tA}(a_t)$ and, by \eqref{eq:def-c},
\[
a_t=z_{t+1}+\alpha_tc_{t+1}.
\]
Since $c_{t+1}\in A(z_{t+1})$, Proposition~\ref{prop:resolvent}(ii) also gives
\[
z_{t+1}=J_{\alpha_{t+1}A}\bigl(z_{t+1}+\alpha_{t+1}c_{t+1}\bigr).
\]
Using the nonexpansiveness of $J_{\alpha_{t+1}A}$,
\[
\norm{d_{t+1}}
=\norm{J_{\alpha_{t+1}A}(a_{t+1})-J_{\alpha_{t+1}A}(z_{t+1}+\alpha_{t+1}c_{t+1})}
\le \norm{a_{t+1}-z_{t+1}-\alpha_{t+1}c_{t+1}}.
\]
We now rewrite the right-hand side. Starting from the definition of $a_{t+1}$ and using \eqref{eq:update-c},
\begin{align*}
a_{t+1}-z_{t+1}-\alpha_{t+1}c_{t+1}
&=\beta_{t+1}(z_0-z_{t+1})-\alpha_{t+1}\bigl(F(z_{t+1})+c_{t+1}\bigr)\\
&=\beta_{t+1}(z_0-z_t-d_t)-\alpha_{t+1}\bigl(F(z_{t+1})-F(z_t)\bigr)
-\alpha_{t+1}\bigl(F(z_t)+c_{t+1}\bigr)\\
&=\Bigl(\frac{\alpha_{t+1}}{\alpha_t}-\beta_{t+1}\Bigr)d_t
-\alpha_{t+1}\bigl(F(z_{t+1})-F(z_t)\bigr)
+\Bigl(\beta_{t+1}-\frac{\alpha_{t+1}}{\alpha_t}\beta_t\Bigr)(z_0-z_t)\\
&=\lambda_td_t-\alpha_{t+1}\bigl(F(z_{t+1})-F(z_t)\bigr)+\varepsilon_t(z_0-z_t).
\end{align*}
Therefore,
\[
\norm{d_{t+1}}
\le
\norm{\lambda_td_t-\alpha_{t+1}(F(z_{t+1})-F(z_t))}+|\varepsilon_t|\norm{z_0-z_t}.
\]
By Lemma~\ref{lem:bounded-iterates}, $\norm{z_0-z_t}\le D$. Also,
\[
\frac{\alpha_{t+1}}{\alpha_t}
=
\sqrt{\frac{t+\gamma}{t+1+\gamma}}
\ge
\frac{t+\gamma}{t+1+\gamma}
\ge
\frac{\gamma}{t+1+\gamma}
=
\beta_{t+1},
\]
so $\lambda_t\ge 0$. Since $F$ is monotone,
\[
\ip{F(z_{t+1})-F(z_t)}{d_t}\ge 0.
\]
Hence
\begin{align*}
&\norm{\lambda_td_t-\alpha_{t+1}(F(z_{t+1})-F(z_t))}^2\\
&=\lambda_t^2\norm{d_t}^2-2\lambda_t\alpha_{t+1}\ip{F(z_{t+1})-F(z_t)}{d_t}+\alpha_{t+1}^2\norm{F(z_{t+1})-F(z_t)}^2\\
&\le \bigl(\lambda_t^2+\alpha_{t+1}^2L^2\bigr)\norm{d_t}^2
=q_t^2\norm{d_t}^2.
\end{align*}
Taking square roots proves the claim.
\end{proof}

\begin{lemma}\label{lem:scalar-bounds}
For every $t\ge 0$,
\[
q_t\le 1-\frac{9}{8(t+1+\gamma)},
\qquad
|\varepsilon_t|\le \frac{\gamma}{(t+\gamma)^2}.
\]
\end{lemma}

\begin{proof}
Set $s:=t+1+\gamma$. Then
\[
\lambda_t=\sqrt{1-\frac{1}{s}}-\frac{\gamma}{s}.
\]
Using $\sqrt{1-x}\le 1-\frac{x}{2}$ for $x\in [0,1]$, we get
\[
\lambda_t\le 1-\frac{\gamma+1/2}{s}.
\]
Therefore
\[
q_t^2\le \left(1-\frac{\gamma+1/2}{s}\right)^2+\frac{1}{s}.
\]
We compare this with $\left(1-\frac{9}{8s}\right)^2$:
\begin{align*}
\left(1-\frac{\gamma+1/2}{s}\right)^2+\frac{1}{s}-\left(1-\frac{9}{8s}\right)^2
&=
\frac{\left(\frac94-2\gamma\right)s+\left(\gamma+\frac12\right)^2-\frac{81}{64}}{s^2}\\
& \le\InParentheses{\left(\frac94-2\gamma\right)(\gamma+1)+\left(\gamma+\frac12\right)^2-\frac{81}{64}}/s^2  \tag{$s\ge \gamma+1$} \\
& = \frac{-64\gamma^2+80\gamma+79}{64 s^2} \\
&\le 0. \tag{$\gamma \ge 2$}
\end{align*}
Hence $q_t^2\le \left(1-\frac{9}{8s}\right)^2$, proving the first inequality.

For the second bound, let $r:=t+\gamma$. Then
\[
\varepsilon_t
=
\frac{\gamma}{r+1}-\frac{\gamma}{r}\sqrt{\frac{r}{r+1}}
=
\frac{\gamma}{r+1}-\frac{\gamma}{\sqrt{r(r+1)}}.
\]
Thus
\[
|\varepsilon_t|
=
\frac{\gamma}{\sqrt{r}(r+1)(\sqrt{r+1}+\sqrt{r})}
\le
\frac{\gamma}{2r(r+1)}
\le
\frac{\gamma}{r^2}
=
\frac{\gamma}{(t+\gamma)^2}. \qedhere
\]
\end{proof}

\begin{proposition}\label{prop:d-decay}
For every $t\ge 0$,
\[
\norm{d_t}\le \frac{E}{t+\gamma}.
\]
\end{proposition}

\begin{proof}
We argue by induction on $t$. The claim is true at $t=0$ by the definition of $E$. Assume it holds at time $t$. Let $r:=t+\gamma$. By Lemmas~\ref{lem:d-recurrence} and~\ref{lem:scalar-bounds},
\[
\norm{d_{t+1}}
\le
\left(1-\frac{9}{8(r+1)}\right)\frac{E}{r}+\frac{\gamma D}{r^2}
=
\frac{E}{r+1}-\frac{E}{8r(r+1)}+\frac{\gamma D}{r^2}.
\]
Since $E\ge 12\gamma D$ and $r\ge \gamma\ge 2$,
\[
\frac{\gamma D}{r^2}
\le
\frac{E}{12r^2}
\le
\frac{E}{8r(r+1)}.
\]
Therefore $\norm{d_{t+1}}\le \dfrac{E}{r+1}=\dfrac{E}{t+1+\gamma}$, which completes the induction.
\end{proof}

\subsection{Step 3: From Iterate Stability to Residual}
\begin{proof}[Proof of Theorem~\ref{thm:main}]
Fix $T\ge 1$. From \eqref{eq:update-c} with $t=T-1$,
\[
d_{T-1}=\beta_{T-1}(z_0-z_{T-1})-\alpha_{T-1}\bigl(F(z_{T-1})+c_T\bigr).
\]
Rearranging gives
\[
F(z_{T-1})+c_T
=
\frac{\beta_{T-1}}{\alpha_{T-1}}(z_0-z_{T-1})-\frac{1}{\alpha_{T-1}}d_{T-1}.
\]
Hence
\begin{align*}
\norm{F(z_T)+c_T}
&\le \norm{F(z_T)-F(z_{T-1})}+\norm{F(z_{T-1})+c_T}\\
&\le L\norm{d_{T-1}}+\frac{1}{\alpha_{T-1}}\norm{d_{T-1}}+\frac{\beta_{T-1}}{\alpha_{T-1}}\norm{z_0-z_{T-1}}.
\end{align*}
Using \eqref{eq:schedule}, Proposition~\ref{prop:d-decay}, and Lemma~\ref{lem:bounded-iterates}, we obtain
\begin{align*}
\norm{F(z_T)+c_T}
&\le L\frac{E}{T-1+\gamma}
+L\sqrt{T-1+\gamma}\cdot \frac{E}{T-1+\gamma}
+\frac{L\gamma}{\sqrt{T-1+\gamma}}\,D\\
&\le \frac{LE}{T-1+\gamma}+\frac{LE}{\sqrt{T-1+\gamma}}+\frac{L\gamma D}{\sqrt{T-1+\gamma}}\\
&\le \frac{L(2E+\gamma D)}{\sqrt{T-1+\gamma}},
\end{align*}
where we used $T-1+\gamma\ge 1$ in the last step.

Since $c_T\in A(z_T)$,
\[
\rtan_{F,A}(z_T)
=\min_{c\in A(z_T)}\norm{F(z_T)+c}
\le \norm{F(z_T)+c_T}
\le \frac{L(2E+\gamma D)}{\sqrt{T-1+\gamma}}.
\]
This proves \eqref{eq:main-bound}. Finally, Lemma~\ref{lem:bounded-iterates} gives
\[
\norm{z_1-z_0}
\le D,
\]
so one may take $E=12\gamma D$. This yields \eqref{eq:main-explicit}. The bound for the natural residual follows from Proposition~\ref{prop:nat-tan}.
\end{proof}

\printbibliography

\appendix

\crefalias{section}{appendix} 

\newpage

\end{document}